\documentclass[11pt]{amsart}

\allowdisplaybreaks

\usepackage{amsmath,amssymb,epsfig,amscd,amsthm,hyperref}

\usepackage{verbatim}

\newtheorem{thm}{Theorem}[section]
\newtheorem{lemma}[thm]{Lemma}

\newtheorem{prop}[thm]{Proposition}

\newtheorem{question}[thm]{Question}
\newtheorem{claim}[thm]{Claim}
\newtheorem{defi}[thm]{Definition}

\numberwithin{equation}{thm}

\theoremstyle{remark}

\newtheorem{remark}[thm]{Remark}

\renewcommand{\bar}[1]{#1\llap{$\overline{\phantom{\rm#1}}$}}

\newcommand{\norm}[1]{\ensuremath{\| #1 \|}}
\newcommand{\lra}{\longrightarrow}

\DeclareMathOperator{\hhat}{{\widehat{h}}}

\newcommand{\prep}{\operatorname{Prep}}
\newcommand{\bA}{{\mathbb A}}
\newcommand{\N}{{\mathbb N}}

\newcommand{\Q}{{\mathbb Q}}

\newcommand{\C}{\mathbb{C}}

\newcommand{\PP}{{\mathbb P}}

\newcommand{\Kbar}{{\bar{K}}}
\newcommand{\Qbar}{\bar \Q}
\newcommand{\Fbar}{{\bar{F}}}

\newcommand{\Gal}{{\rm Gal}}

\newcommand{\cV}{\mathcal{V}}

\newcommand{\into}{\hookrightarrow}

\renewcommand{\l}{\lambda}

\DeclareMathOperator{\id}{{id}}

\newcommand{\bfa}{{\mathbf a}}
\newcommand{\bfh}{{\mathbf h}}

\newcommand{\bfc}{{\mathbf c}}

\newcommand{\bff}{{\mathbf f}}
\newcommand{\bfg}{{\mathbf g}}

\newcommand{\bft}{{\mathbf t}}
\newcommand{\bfu}{{\mathbf u}}

\newcommand{\bfX}{{\mathbf X}}

\newcommand{\cL}{\mathcal{L}}

\newcommand{\bP}{{\mathbb P}}
\newcommand{\cO}{\mathcal{O}}

\newcommand{\bflambda}{\boldsymbol{\l}}

\begin{document}
{\em Dedicated to Joseph Silverman on the occasion of his 60th birthday}
\vspace{1cm}



\title{Unlikely Intersection For Two-Parameter Families of Polynomials}

\author{D.~Ghioca}
\address{
Dragos Ghioca\\
Department of Mathematics\\
University of British Columbia\\
Vancouver, BC V6T 1Z2\\
Canada
}
\email{dghioca@math.ubc.ca}

\author{L.-C.~Hsia}
\address{
Liang-Chung Hsia\\
Department of Mathematics\\
National Taiwan Normal University\\
Taipei, Taiwan, ROC
}
\email{hsia@math.ntnu.edu.tw}

\author{T.~J.~Tucker}
\address{
Thomas Tucker\\
Department of Mathematics\\
University of Rochester\\
Rochester, NY 14627\\
USA
}
\email{ttucker@math.rochester.edu}

\thanks{The first author was partially supported by NSERC. The second
  author was partially supported by NSC Grant  102-2115-M-003-002-MY2 and
  he also acknowledges the support from NCTS.   The third author was
  partially supported by NSF Grants
  DMS-0854839 and DMS-1200749.}



\begin{abstract}
Let $c_1, c_2, c_3$ be distinct complex numbers, and let $d\ge 3$ be an integer. We show that the set of all pairs $(a,b)\in \C\times \C$ such that each $c_i$ is preperiodic for the action of the polynomial $x^d+ax+b$ is not Zariski dense in the affine plane.
\end{abstract}


\maketitle


\section{Introduction}
\label{intro}
The results of this paper are in the context of the \emph{unlikely
  intersections} problem in arithmetic dynamics, and more generally in
arithmetic geometry. The philosophy of the unlikely intersections
principle in arithmetic geometry says that an \emph{event} that is
\emph{unlikely} to occur in a geometric setting must be explained by a
(rigid) arithmetic property. Roughly speaking, in this context, an event is said to be
``unlikely'' when the number of conditions it satisfies is very large
relative to the number of parameters of the underlying space.
For more details, see the Pink-Zilber Conjecture \cite{Pink}, the
various results (such as \cite{BMZ}) in this direction, and also the
beautiful book of Zannier \cite{Zannier}.

At the suggestion of Zannier (whose question was  motivated by \cite{M-Z-1, M-Z-2, M-Z-3}), Baker-DeMarco proved a first result \cite{Matt-Laura} for the unlikely intersection principle this time in arithmetic dynamics. Baker and DeMarco \cite{Matt-Laura} proved that given complex numbers $a$ and $b$, and an integer $d\ge 2$, if there exist infinitely many $t\in \C$ such that both $a$ and $b$ are preperiodic under the action of $z\mapsto z^d+t$, then $a^d=b^d$. Several results followed (see \cite{BD-preprint, prep, metric, Hexia, Hexib}), each time the setting being the following: given two starting points for two families of (one-parameter) algebraic dynamical systems, there exist infinitely many parameters (or more generally, a Zariski dense set of parameters, as considered in \cite{metric}) for which both points are preperiodic at the same time if and only if there is a (precise, global) relation between the two families of dynamical systems and the two starting points.

We note that all results known so far regarding dynamical unlikely
intersection problems are in the context of simultaneous
preperiodicity of two points in a one-parameter families of dynamical
systems, except for~\cite[Theorem~1.4]{metric} which is the first
instance regarding dynamical systems under the action of a family of
endomorphisms of $\bP^2$ parameterized by points of a higher
dimensional variety.   One might ask more generally for dynamical
unlikely intersection problems involving the simultaneous
preperiodicity of $n+1$ points in an $n$-parameter family of dynamical
systems, where $n$ is any positive integer. In this paper, we
consider the general family of polynomial maps on $\PP^1$ of degree
$d\ge 3$ in \emph{normal form} (i.e., polynomials of the form
$z^d+a_{d-2}z^{d-2}+\cdots + a_0$ with parameters $a_{d-2}, \ldots,
a_0$).  The dimension of the space of such maps is $d-1$.  We pose the
following question about simultaneous preperiodicity of $d$ constant
points for polynomials in this family.

\begin{question}
\label{general question}
Let $d\ge 3$ be an integer,  let $c_1,\dots, c_d$ be distinct complex numbers, and let $\bff_\bfa(z) = z^d+a_{d-2}z^{d-2}+\cdots + a_0$ be a family of degree $d$ polynomials  in normal form parametrized by $\bfa = (a_{d-2},\ldots, a_0)\in \bA^{d-1}(\C)$. Is it true that  the set of parameters $\bfa$ such that each $c_i$ is preperiodic under the action of the polynomial $\bff_\bfa$ is not Zariski dense in $\bA^{d-1}$?
\end{question}

\begin{remark}
In the case where $d = 2$, it follows from the main result of~\cite{Matt-Laura} that the set of complex numbers $t$ such that $c_1, c_2$ are preperiodic under the action of the polynomial $\bff_t(z) = z^2 + t$ is Zariski dense in $\bA^1$ if and only if
$c_1^2 = c_2^2.$ Hence, in this case the set of parameters $t$ such that both $c, -c$ (which are distinct if $c\ne 0$) are preperiodic under the action of $\bff_t$ is Zariski dense in the complex affine line.
\end{remark}

In this paper we are able to answer positively the above question when $d=3$; actually, we can  prove a stronger result, as follows.

\begin{thm}
\label{main result}
Let $c_1, c_2, c_3\in \C$ be distinct complex numbers, and let $d\ge 3$ be an integer. Then the set of all pairs $(a_1,a_0)\in \C\times \C$ such that each $c_i$ is preperiodic for the action of $z\mapsto z^d+a_1 z+a_0$ is not Zariski dense in $\bA^2$.
\end{thm}

Theorem~\ref{main result} implies that there are at most finitely many plane curves containing all  pairs of parameters $\bfa = (a_1,a_0)\in \C^2$ such that all $c_i$ (for $i=1,2,3$) are preperiodic under the action of the polynomial $\bff_\bfa(z) = z^d+a_1 z+a_0$. The result is best possible as shown by the following example: if $c\in\C$ is a nonzero number, and $\zeta\in\C$ is a $(d-1)$-st root of unity, then there exist infinitely many $a_1\in\C$ such that $0$, $c$ and $\zeta\cdot c$ are preperiodic for the polynomial $z^d+a_1 z$. The idea is that $c$ is preperiodic for $z^d+a_1 z$ if and only if $\zeta\cdot c$ is preperiodic for $z^d+a_1 z$, and there exist infinitely many $a_1\in\C$ such that $c$ is preperiodic under the action of $z^d+a_1 z$ (by \cite[Proposition~9.1]{prep} applied to the family of polynomials $\bfg_{t}(z):=z^d +t z$ and the starting point $\bfg_{t}(c)=ct+c^d$). There are other more complicated examples showing that the locus of $(a_1,a_0)\in \bA^2$ can be $1$-dimensional. For example, if $d=3$ and $c_1+c_2+c_3=0$, then letting
$$-a_1:=c_1^2+c_1c_2+c_2^2=c_1^2+c_1c_3+c_3^2=c_2^2+c_2c_3+c_3^2,$$
we see that $\bff_\bfa(c_1)=\bff_\bfa(c_2)=\bff_\bfa(c_3)$ and thus there are infinitely many $a_0\in\C$ such that each $c_i$ is preperiodic under the action of $\bff_\bfa$ (for $\bfa=(a_1,a_0)$ with $a_1$ as above).

Also, one cannot expect that Theorem~\ref{main result} can be extended to any $2$-parameter family of polynomials and three starting points. Indeed, for any nonzero $c_1\in\C$ and any $c_2\in\C$, there exists a Zariski dense set of points $(a_1,a_0)\in \bA^2(\C)$ such that the points $c_1$, $-c_1$ and $c_2$ are preperiodic under the action of the polynomial $z^4+a_1z^2+a_0$. We view the $2$-parameter family of polynomials $\bff_{a_1,a_0}(z):=z^d+a_1z+a_0$ as the natural extension of the family of cubic polynomials in normals form, thus explaining why the conclusion of Theorem~\ref{main result} holds for this $2$-parameter family of polynomials, while it fails for other $2$-parameter families of polynomials. Also, the family of polynomials from Theorem~\ref{main result} is the generalization of the $1$-parameter family of polynomials $\bfg_t(z):=z^d+t$ considered by Baker and DeMarco in \cite{Matt-Laura}.

Even though we believe Question~\ref{general question} should be true
in general, we were not able to fully extend our method to the general
case. As we will explain in the next section, there are significant
arithmetic complications arising in the last step of our strategy of
proof when we deal with families of polynomials depending on more than
$2$ parameters.  On the other hand, the last step our proof is
inductive in that it reduces to applying one-dimensional results of
Baker and DeMarco \cite{BD-preprint} to a line in our two-dimensional parameter space.
Thus, we are hopeful that there is a more general inductive argument
that will allow one to obtain a full result in arbitrary dimension.

We describe briefly the contents of our paper. In
Section~\ref{sec:method} we discuss the strategy of our proof and also
state in Theorem~\ref{specialization theorem} a by-product of our proof regarding the variation of the canonical height in an $m$-parameter family of endomorphisms of $\bP^m$ for any $m\ge 2$. In Section~\ref{sec:notation} we introduce our notation and state the necessary background results used in our proof. In Section~\ref{sec:specialization} we prove Theorem~\ref{specialization theorem}, and based on our result in Section~\ref{equidistribution section} we prove a general unlikely intersection statement for the dynamics of polynomials in normal form of arbitrary degree (see Theorem~\ref{equidistribution theorem}). We conclude in Section~\ref{proof section} by proving Theorem~\ref{main result}  using Theorem~\ref{equidistribution theorem}.

\section{Our method of proof}
\label{sec:method}

In the section, we give a sketch of the method used in the proof of our main result. We first prove that if there exist a Zariski dense set of points $\bfa = (a_1,a_0)\in \bA^2(\C)$ such that $c_1,c_2,c_3$ are simultaneously preperiodic under the action of
$$\bff_{\bfa}(z):=z^d+a_1z+a_0,$$
then for \emph{each} point  $\bfa\in \bA^2(\C)$, if any two of the points $c_i$ are preperiodic under the action of $\bff_{\bfa}$, then also the third point $c_i$ is preperiodic. We prove this statement using the powerful equidistribution theorem of Yuan \cite{Yuan} for generic sequences of points of small height on projective varieties $X$ endowed with a metrized line bundle (we also use the function field version of this equidistribution theorem proven by  Gubler \cite{Gubler}). Such equidistribution statements were previously obtained when $X$ is $\bP^1$ by Baker-Rumely \cite{Baker-Rumely} and  Favre-Rivera-Letelier \cite{FRL, favre-rivera}, and when $X$ is an arbitrary curve by Chambert-Loir \cite{CL} and  Thuillier \cite{Thuillier}. Our method is similar to the one employed in \cite{metric} and it extends to polynomials in normal form of arbitrarily degree $d\ge 3$; i.e., by the same technique we prove (see Theorem~\ref{equidistribution theorem}) that given $d$ distinct numbers $c_1,\dots, c_d\in \C$, if there exist a Zariski dense set of points $\bfa=(a_{d-2},\dots, a_0)\in \bA^{d-1}(\C)$ such that each $c_i$ is preperiodic under the action of $\bff_{\bfa}(z),$ then for \emph{each} $\bfa \in \bA^{d-1}(\C)$ such that $d-1$ of the points $c_i$ are preperiodic under the action of $\bff_{\bfa}$, then \emph{all} the $d$ points $c_i$ are preperiodic.

Now, for the $2$-parameter family of polynomials
$\bff_{\bfa}(z):=z^d+a_1z+a_0$, assuming there exists a Zariski dense
set of points $\bfa \in\bA^2(\C)$ such that each $c_i$ (for $i=1,2,3$)
is preperiodic under the action of $\bff_{\bfa}$, we consider the line
$L$ contained in the parameter space $\bA^2$ along which $c_1$ is
fixed by $\bff_{\bfa}(z)$ for each $\bfa\in L(\C)$. Then we have a
$1$-parameter family of polynomials $\bfg_t$  (which is $\bff_{\bfa}$
with $\bfa$ moving along the line $L$), and moreover, for each
parameter $t$, the point $c_2$ is preperiodic for $\bfg_t$ if and only
if $c_3$ is preperiodic for $\bfg_t$.  Applying
\cite[Theorem~1.3]{BD-preprint} (combined with
\cite[Proposition~2.3]{Nguyen}), we obtain that
$\bfg_t^m(c_2)=\bfg_t^m(c_3)$ for some positive integer $m$. This
yields that the starting points $c_i$ are not all distinct, giving a contradiction.

The above argument becomes much more complicated for families of polynomials in normal form parametrized by arbitrary many variables. One could still employ the same strategy and work along the line $L\subset \bA^{d-1}$ along which each of $c_i$, for $i=1,\dots, d-2$ are fixed by $\bff_{\bfa}$ (with $\bfa \in L$). Then  \cite[Theorem~1.3]{BD-preprint} still yields a relation of the form $\bfg_t^m(c_{d-1})=\zeta\cdot \bfg_t^m(c_d)$ (for some root of unity $\zeta$, some positive integer $m$, where $\bfg_t$ is $\bff_{\bfa}$ where $\bfa:=(a_{d-2},\dots, a_0)$ is moving along the line $L$). In turn, this yields a relation between the $c_i$'s which is \emph{a priori} consistent even after redoing the same analysis with the set $\{c_1,\dots, c_{d-2}\}$ replaced by another subset of $\{c_1,\dots, c_d\}$ consisting of $(d-2)$ points. We suspect that in order to derive a contradiction one would have to analyze more general curves in the parameter space along which $(d-2)$ of the points $c_i$ are persistently preperiodic. However this creates additional problems since one would have to prove a generalization of \cite[Theorem~1.3]{BD-preprint} which seems difficult because that result relies (among other ingredients) on a deep theorem of Medvedev-Scanlon \cite{Medvedev-Scanlon} regarding the shape of periodic plane curves under the action of one-variable polynomials acting on each affine coordinate.

As a by-product of our method we obtain a result on the variation of the canonical height in an $m$-parameter family of endomorphisms of $\bP^m$ defined over a product formula field $K$ (for more details on product formula fields, see Section~\ref{sec:notation}). The family of endomorphisms of $\bP^m$ we consider here is a product of the  family of polynomials
$\bff_\bft(z) = z^d + t_1 z^{m-1} + t_2 z^{m-2} + \cdots + t_m$ where $d >  m \ge 2$ and $t_1, \ldots, t_m$ are parameters. Let $\phi := \bff_\bft \times \cdots \times \bff_\bft : \bA^m \to \bA^m$ and extend $\phi$ to a degree $d$ rational map
$\Phi : \bP^m \to \bP^m$. More precisely, let $\bfX := [X_m : X_{m-1} : \cdots : X_0]$  be a homogeneous set of coordinates on $\bP^m$ and let $\Phi_i(\bfX) = X_0^d \bff_\bft(X_i/X_0)$ for $i=m, \ldots, 1$. Then, with respect to the homogeneous coordinates $\bfX$ we have $\Phi(\bfX) = [\Phi_m(\bfX): \cdots : \Phi_1(\bfX) : X_0^d].$ It is easy to verify that $\Phi$ is actually a morphism on $\bP^m$. In the following result, when we specialize our parameter $\bft=(t_1,\dots, t_m)$ to $\bflambda = (\l_1, \ldots, \l_m)\in \bA^m(\Kbar)$, we  write $\Phi_{\bflambda} : \bP^m \to \bP^m $ for the corresponding (specialized) morphism on $\bP^m.$

\begin{thm}
\label{specialization theorem}
Let $d> m\ge 2$ be integers, let $K$ be a number field or a function field of finite transcendence degree over another field, and let $\Phi:\bP^m\lra \bP^m$ be the $m$-parameter family of endomorphisms defined as above.
Let $c_1,\dots, c_m\in\Kbar$ be distinct elements, and let $P:=[c_m:\cdots :c_1:1]\in \bP^m(\Kbar)$. Then for each $\boldsymbol{\l} = (\l_1,\dots, \l_m)\in\bA^m(\Kbar)$, we have the canonical height $\hhat_{\Phi_{\bflambda}}(P)$ constructed with respect to the endomorphism $\Phi_{\bflambda}$ of $\bP^m$ defined over $\Kbar$, and also we have the canonical height $\hhat_\Phi(P)$ constructed with respect to the endomorphism $\Phi$ of $\bP^m$ defined over $\Kbar(t_1,\dots, t_m)$. Then
\begin{equation}
\label{conclusion specialization theorem}
\hhat_{\Phi_{\bflambda}}(P)=\hhat_\Phi(P)\cdot h\left((\l_1,\dots , \l_m)\right)+ O(1),
\end{equation}
where $h\left((\l_1,\dots , \l_m)\right)$ is the Weil height of the point $(\l_1,\dots ,\l_m)\in \bA^m(\Kbar)$ and the constant in $O(1)$ depends  on $c_1, \ldots, c_m$ only.
\end{thm}

It is essential in Theorem~\ref{specialization theorem} that the $c_i$'s are distinct. Indeed, assume $m=2$ and $c_1=c_2=c\in K$. Then for each $\l_1,\l_2\in\Kbar$ satisfying
\begin{equation}
\label{fixed equation}
c^d+\l_1c+\l_2=c,
\end{equation}
we have that the point $P:=[c:c:1]$ is preperiodic under the action of $\Phi_{\l_1,\l_2}$ and thus $\hhat_{\Phi_{\l_1,\l_2}}(P)=0$. On the other hand, $\hhat_{\Phi}(P)=1/d$ (after an easy computation using degrees on the generic fiber of $\Phi$) and thus \eqref{conclusion specialization theorem} cannot hold because there are points $(\l_1,\l_2)\in\bA^2(\Kbar)$ satisfying \eqref{fixed equation} of arbitrarily large height.

Theorem~\ref{specialization theorem} is an improvement of a special
case of Call-Silverman's general result \cite{Call-Silverman} for the
variation of the canonical height in arbitrary families of polarizable
endomorphisms $\Phi_t$ of projective varieties $X$ parametrized by
$t\in T$ (for some base scheme $T$). In the case where the base
variety $T$ is a curve, Call and
Silverman~\cite[Theorem~4.1]{Call-Silverman} have shown that for $P\in
X(\Qbar)$, then
\begin{equation}
  \label{eq:height specialization}
  \hhat_{\Phi_t}(P) = \hhat_{\Phi}(P) h(t) + o(h(t))
  \end{equation}
as we vary $t\in T(\Qbar)$ where $h(\cdot)$ is a height function  associated to a degree one divisor on $T$. Their result generalizes a result of Silverman~\cite{Silverman83} on heights
of families of abelian varieties. In a recent paper~\cite{ingram10},
Ingram improves the error term to $O(1)$ when $T$ is a curve, $X$ is $\bP^1$, and the family of endomorphisms $\Phi$ is totally ramified at infinity (i.e., $\Phi$ is a polynomial mapping).
This result is an analogue of Tate's theorem~\cite{tate83} in  the
setting of arithmetic dynamics.

In order to use Yuan's equidistribution theorem \cite{Yuan} (and same for Gubler's extension \cite{Gubler} to the function field setting) for points of small height to our situation, the
error term in~\eqref{eq:height specialization} needs to be controlled within
$O(1)$ when $\Phi$ is an endomorphism of $\bP^m$ as in Theorem~\ref{specialization theorem}. There are only a few results in the literature when the error term in~\eqref{eq:height specialization} is known to be $O(1)$. Besides  Tate's \cite{tate83} and Silverman's \cite{Silverman83} in the context of elliptic curves and more generally abelian varieties (see also the further improvements of Silverman \cite{Silverman-2, Silverman-3} in the case of elliptic curves), there are only a few known results, all valid for $1$-parameter families (see \cite{ingram10, Ingram-preprint, metric, GM}). To our knowledge, Theorem~\ref{specialization theorem} is the \emph{first} result in the literature where one improves the error term in~\eqref{eq:height specialization} to $O(1)$ for a higher dimensional parameter family of endomorphisms of $\bP^m$.


\section{Notation}
\label{sec:notation}

In this section we setup the notation used in our paper.

\subsection{Maps and preperiodic points}

Let $\Phi:X\lra X$ be a self-map on some set $X$. As always in dynamics, we denote by $\Phi^n$ the $n$-th compositional iterate of $\Phi$ with iteself. We denote by $\id:=\id|_X$ the identity map on $X$.

For any quasiprojective variety $X$ endowed with an endomorphism
$\Phi$, we call a point $x\in X$ \emph{preperiodic} if there exist two
distinct nonnegative integers $m$ and $n$ such that
$\Phi^m(x)=\Phi^n(x)$. If $x=\Phi^n(x)$ for some positive integer $n$, then $x$ is a \emph{periodic} point of \emph{period} $n$. For more details, we refer the reader to the comprehensive book \cite{Silverman07} of Silverman on arithmetic dynamics.

\subsection{Absolute values on product formula fields}

A product formula field $K$  comes
equipped with a standard set $\Omega_K$ of absolute values $|\cdot|_v$ which
satisfy a product formula, i.e,

\begin{equation}
\label{product formula equation 0}
\prod_{v\in \Omega_K} |x|^{N_v}_v = 1\quad\text{ for every $x\in K^*$},
\end{equation}
where $N\colon\Omega_K\to\N$ and $N_v:=N(v)$ (see \cite{lang} for more
details).

The typical examples of product formula fields are
\begin{enumerate}
\item[(1)]  number fields; and
\item[(2)] function fields $K$ of finite transcendence degree over some field $F$.
\end{enumerate}

In the case of function fields $K$, one associates the absolute values in $\Omega_K$ to the irreducible divisors of a smooth, projective variety $\cV$ defined over the constant field $F$ such that $K$ is the function field of $\cV$; for more details, see \cite{lang} and \cite{bg06}. In the special case   $K=F(t_1,\dots, t_m)$, we may take $\cV=\bP^m$.

{\bf As a convention, in order to simplify the notation in this paper, a product formula field is always either a number field or a function field over a constant field.}

Let $K$ be a product formula field. We fix an algebraic closure $\Kbar$ of $K$; if $K$ is a function field of finite transcendence degree over another field $F$ (which we call the constant field), then we also fix an algebraic closure $\Fbar$ of $F$ inside $\Kbar$. Let $v \in \Omega_K$.
Let $\C_v$ be the completion of a fixed algebraic closure of the
completion of $(K,|\cdot |_v)$. When $v$ is an archimedean valuation,
then $\C_v=\C$.  We use the same notation $|\cdot |_v$ to denote the
extension of the absolute value of $(K,|\cdot |_v)$ to $\C_v$ and we
also fix an embedding of $\Kbar$ into $\C_v$.

\subsection{The Weil height}

Let $m\ge 1$, and let $L$ be a finite extension of the product formula field  $K$.
The (naive) Weil height $h(\cdot)$ of any point $P:=[x_m:\cdots :x_0]\in \bP^m(L)$ is defined as
$$h(P)=\frac{1}{[L:K]}\sum_{v\in\Omega_K}N_v\cdot\sum_{\substack{\sigma:L\into \Kbar\\ \sigma|_K=\id}} \log\left(\max\{|x_m|_v, \cdots, |x_0|_v\}\right).$$
So, the above inner sum is over all  possible embeddings of $L$ into $\Kbar$ which fix $K$  pointwise; also one can check that the above definition of height does not depend on the particular choice of the field $L$ containing each $x_i$. We also use the notation $h((x_m,\dots, x_1)):=h([x_m:\cdots :x_1:1])$ to denote the height of the point $(x_m\dots, x_1)$ in the affine space $\bA^m$ embedded in the usual way in $\bP^m$.

In the special case of the function field $K=F(t_1,\dots, t_\ell)$, for a point $P=[x_m:\cdots :x_0]\in\bP^m(K)$, assuming each $x_i\in F[t_1,\dots, t_\ell]$ and moreover, the polynomials $x_i$ are coprime, then $h(P)=\max_{i=0}^m \deg(x_i)$, where $\deg(\cdot )$ is the total degree function on $F[t_1,\dots, t_\ell]$.

\subsection{Canonical heights}

Let $m\ge 1$, and let $f:\PP^m\lra \PP^m$ be an endomorphism of degree
$d\ge 2$.  In \cite{Call-Silverman}, Call and Silverman defined the \emph{global canonical height} $\hhat_f(x)$ for each $x\in \PP^m(\Kbar)$ as
\begin{equation}
\label{definition canonical height function field}
\hhat_f(x)=\lim_{n\to\infty} \frac{h(f^n(x))}{d^n}.
\end{equation}
If $K$ is a number field, then using Northcott's Theorem one deduces
that $x$ is preperiodic for $f$ if and only if $\hhat_f(x)=0$. This
statement does not hold if $K$ is a function field over a constant
field $F$ (which is not a subfield of some $\overline{\mathbb{F}_p}$)
since $\hhat_f(x)=0$ for all $x\in F$ if $f$ is defined over
$F$. However, as proven by Benedetto \cite{Benedetto} and Baker
\cite{Baker_isotrivial}, this is essentially the only counterexample.

\subsection{Canonical height over function fields}

In order to state the results of Baker and Benedetto, we first define isotrivial polynomials.

\begin{defi}
\label{isotrivial}
We say a polynomial $f \in K[z]$ is \emph{isotrivial} over $F$ if there
exists a linear $\ell
\in \Kbar[z]$ such that $\ell\circ f \circ \ell^{-1} \in \Fbar[z]$.
\end{defi}

Benedetto proved that a non-isotrivial polynomial has nonzero
canonical height at its non-preperiodic points
\cite[Thm.~B]{Benedetto}.  As stated, Benedetto's result applies only
to function fields of transcendence dimension one, but the proof
extends easily to function fields of any transcendence dimension.
Baker \cite{Baker_isotrivial} later generalized the result to the case
of rational functions over arbitrary product formula fields.

\begin{lemma}[Benedetto \cite{Benedetto}, Baker \cite{Baker_isotrivial}]
\label{Benedetto}
  Let $f\in K[z]$ with $\deg(f)\ge 2$, and let
$x\in \Kbar$. If $f$ is non-isotrivial over $F$, then
$ \hhat_{f}(x) = 0$ if and only if $x$ is preperiodic for~$f$.
\end{lemma}

A crucial observation for our paper is that a polynomial in normal form is isotrivial if and only if it is defined over the constant field; the following result is proven in \cite[Lemma~10.2]{prep}.
\begin{prop}
\label{normal form proposition}
Let $f\in K[z]$ be a polynomial in normal form. Then $f$ is isotrivial over $F$ if and only if $f\in \Fbar[z]$.
\end{prop}

%
%
%
%
%
%


\section{Proof of the specialization theorem}
\label{sec:specialization}

In this section we prove Theorem~\ref{specialization theorem}. So, we work with the following setup:
\begin{itemize}
\item $d>m\ge 2$ are integers.
\item $K$ is a product formula field of characteristic $0$.
\item For algebraically independent variables $t_1,\dots, t_m$ we define
$$\bff(z):=z^d + t_1z^{m-1} + \cdots + t_{m-1}z+t_m.$$
Let $\Phi:\bP^m\lra \bP^m$ be the map on $\bP^m$  defined by
$$\quad\; \Phi([X_m:\cdots :X_1:X_0]) = \left[ X_0^d\bff\left(\frac{X_m}{X_0}\right) : \cdots :  X_0^d\bff\left(\frac{X_1}{X_0}\right): X_0^d\right].$$
It is straightforward to verify that $\Phi$ is a morphism on $\bP^m$ over $K(t_1,\dots, t_m)$.
\item When we specialize each $t_i$ to some $\l_i\in \Kbar$, we use the notation $\bff_{\bflambda}(z):=z^d + \l_1z^{m-1}+\cdots + \l_{m-1}z+\l_m$ and  $\Phi_{\bflambda}$ to denote the corresponding specialized polynomial and endomorphism of $\bP^m$ respectively, where $\bflambda = (\l_1,\ldots, \l_m)\in \bA^m(\Kbar)$.
\item Let $\bfc:=(c_m,\ldots , c_1) \in \bA^m(K)$ where the $c_i$'s are distinct. We denote the point $[c_m:\ldots : c_1: 1]\in \bP^m(K)$ by $\tilde{\bfc}$.
\item For each $\bflambda :=(\l_1,\ldots,\l_m) \in\bA^m(\Kbar)$, we let $\hhat_{\Phi_{\bflambda}}$ be the canonical height corresponding to the endomorphism $\Phi_{\bflambda}$ defined over $\Kbar$; also we let $\hhat_\Phi$ be the canonical height corresponding to the endomorphism $\Phi$ defined over the function field $K(t_1,\dots, t_m)$.
\item Let $v\in \Omega_K.$ The ($v$-adic) {\em norm} of a point  $P = [x_m:\cdots:x_1: x_0]\in \bP^m(\C_v)$ with $x_0\ne 0$ is defined by $\norm{P} := \max \{|x_m/x_0|_v, \ldots, |x_0/x_0|_v\}$. Also, for a point
$Q = (a_m, \ldots, a_1)\in \bA^m(\C_v)$, we define its norm $\norm{Q}_v := \max\{|a_m|_v, \ldots, |a_1|_v, 1\}$; it is clear that $\norm{Q}_v=\norm{\tilde{Q}}_v$, where $\tilde{Q}:=[a_m:\cdots:a_1:1]$.  For a polynomial
$$g(x_1, \ldots, x_m) = \sum\,a_{i_1, \ldots, i_m} x_1^{i_1}\cdots x_m^{i_m} \in \C_v[x_1, \ldots, x_m],$$ the norm of $g$ is defined by $\norm{g}_v : = \max_{i_1,\ldots,i_m}\{|a_{i_1,\ldots,i_m}|_v\}$. Similarly, for a morphism
    $\Psi = [\psi_m : \ldots : \psi_1: \psi_0] : \bP^m \to \bP^m$ over $\C_v$, we set $\norm{\Psi}_v := \max_i\{\norm{\psi_i}_v\}/\norm{\psi_0}_v$.
\item By abuse of the notation, we simply write $\Phi^n(\bfc)$ for $\Phi^n(\tilde{\bfc})$ and similarly, we let $\hhat_{\Phi_{\bflambda}}(\bfc)$ denote the canonical height of the point $\tilde{\bfc}$ etc.
\end{itemize}
\medskip

For each $n\ge 0$ and each $i=1,\dots, m$ we define $A_{n,i}(t_1,\dots, t_m)$ such that
$$\Phi^n(\bfc):=[A_{n,m}:\cdots :A_{n,1}:1].$$
Then $A_{0,i}=c_i$ for each $i=1,\dots, m$ and for general $n\ge 0$:
$$A_{n+1,i}(t_1,\dots, t_m):=\bff(A_{n,i}(t_1,\dots, t_m)).$$

It is easy to see that the total degree $\deg(A_{n,i})$ in the variables $t_1,\dots, t_m$ equals $d^{n-1}$; so $\hhat_\Phi(\bfc)=\frac{1}{d}$ (see \eqref{definition canonical height function field}).  Therefore \eqref{conclusion specialization theorem} reduces to proving
\begin{equation}
\label{conclusion specialization theorem 2}
\hhat_{\Phi_{\bflambda}}(\bfc)=\frac{h(\bflambda)}{d} + O(1), \quad \bflambda = (\l_1,\dots, \l_m)\in \bA^m(\Kbar).
\end{equation}

Note that by our convention mentioned above, we have
$$\norm{\Phi^n_{\bflambda}(\bfc)}_v :=\max\{1,|A_{n,1}(\l_1,\dots, \l_m)|_v,\cdots , |A_{n,m}(\l_1,\dots, \l_m)|_v\}.$$
To ease the notation, in the following discussion we simply denote $A_{n,i}(\l_1, \dots, \l_m)$ by $A_{n,i}$ when $\l_1, \dots, \l_m$ are fixed.

Using the definition of $\Phi^n(\bfc)$ which yields that each $A_{n,i}$ has total degree  $d^{n-1}$ in $\bflambda$ and also degree at most $d^n$ in $\bfc$, we obtain an upper bound for $\norm{\Phi_{\bflambda}^n(\bfc)}_v$ when $v$ is a nonarchimedean place of $K.$
\begin{equation}
\label{eqn:upper bound for phi}
\norm{\Phi^n_{\bflambda}(\bfc)}_v \le  \norm{\bflambda}^{d^{n-1}}_v \norm{\bfc}^{d^n}_v
\end{equation}
We first observe that from the definition of the $v$-adic norm of a point, we have that always the $v$-adic norm of a point is at least equal to $1$, i.e.,
\begin{equation}
\label{norm is at least 1}
\norm{\bfc}_v\ge 1\text{ and }\norm{\bflambda}_v\ge 1.
\end{equation}

Next we prove a couple of easy lemmas.
\begin{lemma}
\label{almost all v are 0}
Let $\bflambda = (\l_1,\dots, \l_m)\in\bA^m(\Kbar)$, and let $|\cdot |_v$ be a nonarchimedean absolute value such that $\norm{\bfc}_v = 1$ and $\norm{\bflambda}_v = 1$. Then $\norm{\Phi_{\bflambda}^n(\bfc)}_v=1$ for each $n\ge 0$.
\end{lemma}

\begin{proof}
The result follows using~\eqref{eqn:upper bound for phi} and also that (just as for any point; see for example, \eqref{norm is at least 1})  $\norm{\Phi^n_{\bflambda}(\bfc)}_v\ge 1$ for every $n\ge 0$.
\end{proof}

\begin{lemma}
\label{almost all v work}
Let $|\cdot |_v$ be a nonarchimedean absolute value such that
\begin{itemize}
\item[(i)] $\norm{\bfc}_v =1$; and
\item[(ii)] $|c_i-c_j|_v=1$ for each $1\le i< j\le m$.
\end{itemize}
Then for each $\bflambda \in \bA^m(\Kbar)$, and for each $n\ge 1$ we have
\begin{equation}
\label{conclusion lemma almost all v work}
\norm{\Phi^n_{\bflambda}(\bfc)}_v = \norm{\bflambda}^{d^{n-1}}_v.
\end{equation}
\end{lemma}

\begin{proof}
First, we note by Lemma~\ref{almost all v are 0} (see condition~(i) above) that if $\norm{\bflambda}_v= 1$, then $\norm{\Phi^n_{\bflambda}(\bfc)}_v=1$ as claimed in the above conclusion. So, from now on, we assume that $\norm{\bflambda}_v >1$. Hence
$|\l_i|_v>1$ for some $i=1,\dots, m$.

For $n=1$, we note that
\begin{equation}
\label{determinant 1}
\left\{\begin{array}{ccc}
c_m^{m-1}\l_1+ \cdots + c_m\l_{m-1} + \l_m & = & A_{1,m}-c_m^d\\
\cdots \cdots \cdots \cdots \cdots & \cdots & \cdots \cdots \\
c_1^{m-1}\l_1 + \cdots + c_1\l_{m-1} + \l_m & = & A_{1,1} - c_1^d
\end{array} \right.
\end{equation}
seen as a system with unknowns $\l_1,\dots, \l_m$ has the determinant equal with a van der Monde   determinant which is a $v$-adic unit (see condition~(ii) above). Therefore,  using also that $|c_i|_v\le 1$, we get that
$$|\l_i|_v \le \max\{1, |A_{1,1}|_v,\dots , |A_{1,m}|_v\}\quad i = 1, \ldots, m. $$
Thus, $\norm{\bflambda}_v \le \norm{\Phi_{\bflambda}(\bfc)}_v$. Combining this last inequality  with~\eqref{eqn:upper bound for phi} for $n=1$,
we conclude that $\norm{\Phi_{\bflambda}(\bfc)}_v = \norm{\bflambda}_v.$

Now, for $n>1$, we argue by induction on $n$. So, assume \eqref{conclusion lemma almost all v work} holds for $n=k\ge 1$, and we prove the same equality holds for $n=k+1$. By induction hypothesis and using the fact that $d>m$, we have that for each $j=0,\dots, m-1$,
\begin{equation}
\label{less valuation k+1}
|A_{k,i}^j\l_{m-j}|_v\le \norm{\bflambda}_v^{j\cdot d^{k-1}} \, \norm{\bflambda}_v < \norm{\bflambda}_v^{d^k}.
\end{equation}
For the last inequality we also use the fact that $d>m>j$ and that $\norm{\bflambda}_v>1$.  So,
$$|A_{k+1,i}|_v\le \max\left\{|A_{k,i}|_v^d , \max_{j=0}^{m-1} |A_{k,i}^j \l_{m-j}|_v\right\}\le \norm{\bflambda}_v^{d^k}.$$
On the other hand, since $\norm{\Phi_{\bflambda}^k(\bfc)}_v = \norm{\bflambda}^{d^{k-1}}_v > 1$ by the induction hypothesis, there exists some $i=1,\dots, m$ such that $|A_{k,i}|_v=\norm{\bflambda}_v^{d^{k-1}}$ and so, for that index $i$ (using \eqref{less valuation k+1}), we have $|A_{k+1,i}|_v=|A_{k,i}|_v^d=\norm{\bflambda}_v^{d^{k}}$, as claimed.
\end{proof}

Let $S\subset \Omega_K$  consist of all the archimedean places of $K$ and all the places $v$ which do \emph{not} satisfy at least one of the two conditions (i) and (ii) from Lemma~\ref{almost all v work}.
It is clear that the set $S$ is finite (note that $c_i\ne c_j$ for $i\ne j$ and thus condition~(ii) from Lemma~\ref{almost all v work} is satisfied by all but finitely many places $v$).

\begin{lemma}
\label{reduction to finitely many places}
Let $\bflambda = (\l_1,\dots, \l_m) \in\bA^m(\Kbar)$, and let $L$ be a finite, normal extension of $K$ containing $\l_1,\dots, \l_m$. Then we have
\begin{gather*}
\hhat_{\Phi_{\bflambda}}(\bfc) - \hhat_\Phi(\bfc)\cdot h(\Phi_{\bflambda}(\bfc)) \\
  = \sum_{v\in S}\frac{N_v}{[L:K]}\cdot \sum_{\sigma\in\Gal(L/K)} \left(\lim_{n\to\infty} \frac{\log \norm{\Phi_{\sigma(\bflambda)}^n(\bfc)}_v}{d^n} - \frac{\log \norm{\Phi_{\sigma(\bflambda)}(\bfc)}_v}{d}\right).
  \end{gather*}
\end{lemma}

\begin{proof}
We have
$$\hhat_{\Phi_{\bflambda}}(\bfc) = \lim_{n\to\infty}\sum_{v\in\Omega_K}\frac{N_v}{[L:K]}\cdot \sum_{\sigma\in \Gal(L/K)}\frac{\log \norm{\Phi_{\sigma(\bflambda)}^n(\bfc)}_v}{d^n}.$$
Lemma~\ref{almost all v are 0} yields that for all but finitely many absolute values $|\cdot |_v$ we have that $\norm{\Phi_{\sigma(\bflambda)}^n(\bfc)}_v=1$ for each $\sigma\in\Gal(L/K)$. So, we can interchange the above limit with the sum formula and get
$$\hhat_{\Phi_{\bflambda}}(\bfc)=\sum_{v\in \Omega_K} \frac{N_v}{[L:K]}\cdot \sum_{\sigma\in\Gal(L/K)}\lim_{n\to\infty} \frac{\log \norm{\Phi_{\sigma(\bflambda)}^n(\bfc)}_v}{d^n}.$$
Lemma~\ref{almost all v work} finishes the proof of Lemma~\ref{reduction to finitely many places}
\end{proof}

The next result is the key technical step which allows us to deal with the \emph{potentially bad} places $v\in S$ for proving \eqref{conclusion specialization theorem 2}.

\begin{lemma}
\label{analysis of bad places}
Let $v\in S$. There exists a constant $C(v,\bfc)$ depending only on the absolute value $|\cdot |_v$ and on the point $\bfc$ such that for each $\bflambda = (\l_1,\dots, \l_m)\in\bA^m(\Kbar)$, and for each positive integers $n_2>n_1$, we have
$$\left|\frac{\log \norm{\Phi_{\bflambda}^{n_2}(\bfc)}_v}{d^{n_2}} - \frac{\log \norm{\Phi_{\bflambda}^{n_1}(\bfc)}_v}{d^{n_1}}\right| < \frac{C(v,\bfc)}{d^{n_1}}.$$
\end{lemma}

\begin{proof}
Since we will fix the place $v$ and $\l_1,\dots,\l_m$ and $|\cdot |_v$,  we  simply denote $M_n:=\norm{\Phi_{\bflambda}^n(\bfc)}_v$. Similarly, as stated above,  we will use the notation $A_{n,i}:=A_{n,i}(\l_1,\dots, \l_m)$ for $i=1,\dots, m$.
We split the analysis based on whether there exists at least one $\l_i$ with large absolute value or not.
\begin{claim}
\label{again in a finite ball}
Let $L$ be any real number larger than $1$. If $\norm{\bflambda}_v\le L$, then
$$\frac{1}{2^dL^d}\le \frac{ M_{n+1}}{M_n^d}\le (m+1)L,$$
for all $n\ge 1$.
\end{claim}

\begin{proof}[Proof of Claim~\ref{again in a finite ball}.]
Now, by definition, $M_n\ge 1$; so,  using also that $L>1$,  we get that for each $i=1,\dots, m$ we have
$$|A_{n+1,i}|_v\le |A_n|^d_v+\sum_{i=1}^m |\l_i|_v\cdot |A_n|^{m-i}_v\le (m+1)L\cdot M_n^d.$$
Thus $M_{n+1} = \max\{1, |A_{n+1,1}|_v, \ldots, |A_{n+1,m}|_v\} \le  (m+1)L\cdot M_n^d.$
This proves the existence of the upper bound in Claim~\ref{again in a finite ball}.

For the proof of the existence of the lower bound, we split our analysis into two cases:

{\bf Case 1.} $M_n\le 2L$.

In this case, using that $M_{n+1}\ge 1$, we immediately obtain that
\begin{equation}
\label{1}
\frac{M_{n+1}}{M_n^d}\ge \frac{1}{2^dL^d}.
\end{equation}

{\bf Case 2.} $M_n>2L$.

Let $j\in\{1,\dots, m\}$ such that $|A_{n,j}|_v=M_n$; then
\begin{align*}
 |A_{n+1,j}|_v & = |A_{n,j}^d+\sum_{i=1}^m A_{n,j}^{m-i}\l_i|_v\\
& \ge |A_{n,j}|_v^d - \sum_{i=1}^m |A_{n,j}|_v^{m-i}\cdot |\l_i|_v\\
& \ge |A_{n,j}|_v^d\cdot \left( 1- \sum_{i=1}^m \frac{|\l_i|_v}{|A_{n,j}|_v^{d-m+i}} \right)\\
& \ge M_n^d\cdot \left( 1- \sum_{i=1}^m \frac{L}{M_n^{i+1}}\right)\; \text{since $\norm{\bflambda}_v \le L, |A_{n,j}|_v = M_n$ and $m < d$,}\\
& \ge M_n^d\cdot \left( 1- \frac{1}{2}\right)\; \text{since $M_n > 2 L$,}\\
& \ge \frac{1}{2}\cdot M_n^d.
\end{align*}

Since $M_{n+1}\ge |A_{n+1,j}|_v$, the above inequality coupled with inequality~\eqref{1} yields the lower bound from the conclusion of Claim~\ref{again in a finite ball}.
\end{proof}

We continue the proof of Lemma~\ref{analysis of bad places}. We solve for the $\l_i$'s in terms of the $A_{1,i}$'s from the system \eqref{determinant 1} and obtain that for each $k=1,\dots, m$ we have
\begin{equation}
\label{solving for l_i}
\l_k = \frac{Q_{k,0}(c_1,\dots, c_m)+\sum_{i=1}^m Q_{k,i}(c_1,\dots, c_m)\cdot A_{1,i}}{\prod_{1\le i<j\le m}(c_i-c_j)},
\end{equation}
where each $Q_{k,i}(X_1,\dots, X_m)$ is a polynomial of degree at most $d$ in each variable $X_i$.

Let $L_0$ be a real number satisfying the following inequalities:
\begin{itemize}
\item[(1)] $L_0\ge (m+1)(d+1)^m\cdot \norm{\bfc}_v^{dm}$;
\item[(2)] $L_0$ is larger than the $v$-adic absolute value of each coefficient of each $Q_{k,i}$, for $k=1,\dots, m$ and  $i=0,\dots, m$;
\item[(3)] $L_0\ge \frac{1}{ {\displaystyle \prod_{1\le i<j\le m}|c_i-c_j|_v}}$.
\end{itemize}

\begin{claim}
\label{outside the finite ball}
Let $L\ge 4L_0^{6}$ be a real number. Then for each $\bflambda=(\l_1,\dots, \l_m)\in\bA^m(\C_v)$ such that  $\norm{\bflambda}_v>L$, we have
$$\frac{1}{2}\le \frac{M_{n+1}}{M_n^d}\le 2,$$
for each $n\ge 1$.
\end{claim}

\begin{proof}[Proof of Claim~\ref{outside the finite ball}.]
First we prove that $M_1\ge \frac{\norm{\bflambda}_v}{L_0^3}$. Note that by our choice of $L_0$, the triangle inequality gives  $$|Q_{k,i}(c_1,\dots,c_m)|_v\le L_0 (d+1)^m \|\bfc\|_v^{dm}.$$
Using \eqref{solving for l_i} (coupled with inequalities (1)-(3) for $L_0$), we get that
\begin{equation}
\label{l_i compared with A_i}
|\l_k|_v \le \frac{(m+1)L_0\cdot (d+1)^m\cdot\|\bfc\|^{dm}\cdot M_1}{\prod_{1\le i<j \le m}\,|c_i - c_j|_v}\le L_0^3M_1.
\end{equation}
Thus, $\norm{\bflambda}_v  = \max\{1, |\l_1|_v, \ldots, |\l_m|_v\} \le L_0^3 M_1$ as first claimed.

We prove by induction that for each $n\ge 1$, we have $M_n\ge \frac{\norm{\bflambda}_v}{L_0^3}$.
We already established the inequality for $n=1$. Now, assume $M_n\ge \frac{\norm{\bflambda}_v}{L_0^3}$ and we show next that
\begin{equation}
\label{5}
M_{n+1}\ge \frac{M_n^d}{2}\ge \frac{\norm{\bflambda}_v}{L_0^3}>2.
\end{equation}
Note that the last inequality from \eqref{5} follows from the fact that  $\norm{\bflambda}_v>L\ge 4L_0^6>2L_0^3$.
Without loss of generality, we may assume $|A_{n,i}|_v=M_n$ for some $i$, and so by the induction hypothesis $|A_{n,i}|_v = M_n \ge \frac{\norm{\bflambda}_v}{L_0^3}\ge 2$. Now,
\begin{align*}
M_{n+1} & \ge |A_{n+1,i}|_v \\
& \ge |A_{n,i}|_v^d - \sum_{j=1}^m |\l_j|_v\cdot |A_{n,i}|^{m-j}_v  \\
& \ge |A_{n,i}|_v^d\cdot \left(1 - \sum_{j=1}^m \frac{|\l_j|_v}{|A_{n,i}|_v^{d-m+j}}\right)\\
& \ge |A_{n,i}|_v^d\cdot \left(1 - \frac{2\norm{\bflambda}_v}{|A_{n,i}|_v^2}\sum_{j=1}^m \frac{|A_{n,i}|_v^2}{2 |A_{n,i}|_v^{d-m+j}}\right) \;\text{since $|\l_i|_v \le \norm{\bflambda}_v$,} \\
& \ge |A_{n,i}|_v^d \cdot \left(1 - \frac{2\norm{\bflambda}_v}{|A_{n,i}|_v^2}\right)\text{ since $|A_{n,i}|_v \ge 2$,}\\
& \ge |A_{n,i}|_v^d \cdot \left(1- \frac{2L_0^6}{\norm{\bflambda}_v}\right)\text{ by the induction hypothesis,}\\
& \ge |A_{n,i}|_v^d\cdot \left(1- \frac{2L_0^6}{L}\right)\text{ by the hypothesis of Claim~\ref{outside the finite ball}}\\
& \ge \frac{|A_{n,i}|_v^d}{2}=\frac{M_n^d}{2}\\
& \ge \frac{\norm{\bflambda}_v^d}{2L_0^{3d}}\\
& \ge \frac{\norm{\bflambda}_v}{L_0^3}\;\; \text{because $\norm{\bflambda}_v>L\ge 4L_0^6$.}
\end{align*}
 We note that the first inequality from \eqref{5} already yields the lower bound from the conclusion of Lemma~\ref{outside the finite ball}.

Next we prove that for all $n\ge 1$, we have
\begin{equation}
\label{4}
\frac{M_{n+1}}{M_n^d}\le 2.
\end{equation}
Again, without loss of generality, we may assume $|A_{n+1,i}|_v=M_{n+1}$. Then using inequality \eqref{5}, we get
\begin{align*}
M_{n+1}=|A_{n+1,i}|_v & \le |A_{n,i}|_v^d + \sum_{j=1}^m|\l_j|_v\cdot |A_{n,i}|^{m-j}_v \\
& \le M_n^d + \sum_{j=1}^m |\l_j|_v\cdot M_n^{m-j}\\
& \le M_n^d\cdot \left(1 + \sum_{j=1}^m \frac{|\l_j|_v}{M_n^{d-m+j}} \right)\\
& \le M_n^d\cdot \left(1 +\frac{2\norm{\bflambda}_v}{M_n^2}\right)\quad \text{ because $M_n\ge 2$}\\
& \le M_n^d\cdot \left(1 + \frac{2L_0^6}{\norm{\bflambda}_v}\right)\quad \text{by induction hypothesis}\\
& \le M_n^d\cdot \left(1+\frac{2L_0^6}{L}\right)\quad \text{since $\norm{\bflambda}_v > L$}\\
& \le 2M_n^d \quad\text{since $L \ge 4 L_0^6$}.
\end{align*}
This concludes the proof of Lemma~\ref{outside the finite ball}.
\end{proof}

Claims~\ref{again in a finite ball} and \ref{outside the finite ball} yield that there exists a constant $C>1$ (depending on $v$ and $\bfc$) such that
\begin{equation}
\label{almost telescoping}
\frac{1}{C}\le \frac{M_{n+1}}{M_n^d}\le C,
\end{equation}
for each $n\ge 1$.
An easy  telescoping sum after taking the logarithm of the inequalities from \eqref{almost telescoping} finishes the proof of Lemma~\ref{analysis of bad places}.
\end{proof}

An immediate corollary of Lemma~\ref{analysis of bad places} (for $n_1=1$) is the following result.
\begin{lemma}
\label{analysis of bad places 1}
Let $v\in S$. There exists a constant $C(v,\bfc)$ depending only on the absolute value $|\cdot |_v$ and on the point $\bfc$ such that for each $\bflambda = (\l_1,\dots, \l_m)\in\bA^m(\Kbar)$, and for each positive integers $n$, we have
$$\left|\lim_{n\to\infty}\frac{\log \norm{\Phi_{\bflambda}^n(\bfc)}_v}{d^{n}} - \frac{\log \norm{\Phi_{\bflambda}(\bfc)}_v}{d}\right| < \frac{C(v,\bfc)}{d}.$$
\end{lemma}

The next result is an easy consequence of the height machine.

\begin{lemma}
\label{reduction to M_1}
There exists a constant $C(\bfc)$ depending only on the point $\bfc$ such that
$$\left|h\left(\Phi_{\bflambda}(\bfc)\right)  - h(\bflambda)\right|\le C(\bfc),$$
for each $\bflambda \in\bA^m(\Kbar)$.
\end{lemma}

\begin{proof}
Recall that $\bff(x) = x^d + t_1 x^{m-1} + \cdots + t_{m-1} x + t_m.$
We consider the linear transformation $\Psi:\bP^m\lra \bP^m$ defined by
$$\Psi(P) = [T_0 \bff(c_m): \ldots : T_0 \bff(c_1):T_0]$$
where $P = [T_m : \ldots: T_1: T_0]$ and $t_i = T_i/T_0$ for $i=1,\ldots, m.$ Since the $c_i$'s are distinct, the map $\Psi$ is an automorphism of $\bP^m$. So, there exists a constant $C(\bfc)$ (see \cite{bg06}) depending only on the constants $c_i$ such that
\begin{equation}
\label{constant for automorphism}
\left|h(\Psi([\l_m:\cdots : \l_1:1])) - h([\l_m:\cdots :\l_1:1])\right|\le C(\bfc).
\end{equation}
Because $\Phi_{\bflambda}(\bfc)=\Psi([\l_m:\cdots :\l_1:1])$ and $h(\bflambda)=h([\l_m:\cdots :\l_1:1])$, we obtain the conclusion of Lemma~\ref{reduction to M_1}.
\end{proof}

\noindent We are ready to prove Theorem~\ref{specialization theorem}.
\begin{proof}[Proof of Theorem~\ref{specialization theorem}.]
Let $L$ be a finite, normal extension of $K$ containing each $\l_i$ (for $i=1,\dots, m$).
Combining Lemmas~\ref{reduction to finitely many places}, \ref{analysis of bad places 1} and  \ref{reduction to M_1} yields that
$$\left|\hhat_{\Phi_{\bflambda}}(\bfc)-\hhat_\Phi(P)\cdot h(\bflambda)\right|
\le   \frac{1}{d}\cdot \left|h\left(\Phi_{\bflambda}(\bfc)\right) - h(\bflambda)\right| +
\left|\hhat_{\Phi_{\bflambda}}(\bfc) - \frac{h\left(\Phi_{\bflambda}(\bfc)\right)}{d}\right|
$$
\begin{align*}
& \le C(\bfc) + \sum_{v\in S}\frac{N_v}{[L:K]}\cdot \sum_{\sigma\in \Gal(L/K)}\left| \lim_{n\to\infty} \frac{\log \norm{\Phi_{\sigma(\bflambda)}^n(\bfc)}_v}{d^n} - \frac{\log \norm{\Phi_{\sigma(\bflambda)}(\bfc)}_v}{d}\right| \\
& \le C(\bfc) + \sum_{v\in S}\frac{N_v}{[L:K]}\cdot [L:K]\cdot C(v,\bfc) \\
&\le C(\bfc)+\sum_{v\in S}N_vC(v,\bfc)
\end{align*}
as desired.

\end{proof}


\section{Simultaneously preperiodic points for polynomials of arbitrary degree}
\label{equidistribution section}

We retain the notation used in Section~\ref{sec:specialization}.
In this section we prove the following result.

\begin{thm}
\label{equidistribution theorem}
Let $K$ be a number field, or a function field of finite transcendence degree over $\Qbar$, let $d>m\ge 2$ be integers, and let
$$\bff(z):=z^d+ t_1z^{m-1}+\cdots + t_{m-1}z+t_m$$
be an $m$-parameter family of polynomials of degree $d$. For each point $\bflambda = (\l_1,\dots, \l_m)$ of $\bA^m(\Kbar)$ we let $\bff_{\bflambda}$ be the corresponding polynomial defined over $\Kbar$ obtained by specializing each $t_i$ to $\l_i$ for $i=1,\dots, m$.

Let $c_1,\dots, c_{m+1}\in K$ be distinct elements. Let $\prep(c_1, \ldots, c_{m+1})$ be the set consisting of parameters $\bflambda\in \bA^m(\Kbar)$ such that $c_i$ is preperiodic for $\bff_{\bflambda}$ for each $i=1, \ldots, m+1$.
If $\prep(c_1,\ldots, c_{m+1})$ is Zariski dense  in $\bA^m(\Kbar)$ then  the following holds: for each $\bflambda\in \bA^m(\Kbar)$, if $m$ of the points $c_1,\dots, c_{m+1}$ are preperiodic under the action of $\bff_{\bflambda}$, then all $(m+1)$ points are preperiodic under the action of $\bff_{\bflambda}$.
\end{thm}

\begin{remark}
\label{remark Zariski dense}
We note that Theorem~\ref{equidistribution theorem} does not hold if the $c_i$'s are not all distinct. This can be seen for example when $d=3$ and $m=2$ by considering starting points $c_1\ne c_2= c_3$. One can show that in this case, $\prep(c_1,c_2,c_3)=\prep(c_1,c_2)$ is Zariski dense in $\bA^2$. Indeed, otherwise there are finitely many irreducible plane curves $C_i$ (for $i=1,\dots, \ell$) containing all points from $\prep(c_1,c_2)$. Then consider a preperiodicity portrait $(m_1,n_1)$ for the point $c_1$ which is not identically realized along any of the curves $C_i$; the existence of such a portrait is guaranteed by \cite[Theorem~1.3]{GNT}. Then there exists a curve $C:=C_{(m_1,n_1)}\subset \bA^2$ such that for each $(a_1,a_0)\in C(\Kbar)$, the preperiodicity portrait of $c_1$ under $\bff_\bfa(z):=z^3+a_1z+a_0$ is $(m_1,n_1)$. Another application of \cite[Theorem~1.3]{GNT} yields the existence of infinitely many points $(a_1,a_0)\in C(\Kbar)$ such that $c_2$ is preperiodic under the action of $\bff_\bfa$. But this means that $C$ must be contained in the Zariski closure of $\prep(c_1,c_2)$ contradicting the fact that $C$ is not one of the curves $C_i$ for $i=1,\dots, \ell$.
\end{remark}

So, Theorem~\ref{equidistribution theorem} yields that if there exists a Zariski dense set of $m$-tuples $\bflambda = (\l_1,\dots, \l_m)\in \bA^m(\Kbar)$ such that each $c_i$ is preperiodic under the action of $\bff_{\bflambda}$, then something \emph{quite unlikely} holds: for any specialization polynomial $\bff_{\bflambda}$, if $m$ points $c_i$ are preperiodic, then \emph{all} $(m+1)$ points $c_i$ are preperiodic under the action of $\bff_{\bflambda}$. As discussed in Section~\ref{intro} (see also Remark~\ref{remark Zariski dense} and \cite{GNT}), it is expected that there are \emph{many} specializations $\bff_{\bflambda}$ such that $c_1,\dots, c_{m}$ are preperiodic under the action of $\bff_{\bflambda}$. So, Theorem~\ref{equidistribution theorem} yields that under the given conclusion, for each of these \emph{many} specializations, \emph{all} $(m+1)$ points $c_i$ are preperiodic. We expect that such a conclusion should actually  yield a contradiction, and in the next Section we are able to prove this in the case $m=2$.

The main ingredient in proving Theorem~\ref{equidistribution theorem} is the  powerful equidistribution theorem for points of small height with respect to metrized ad\'elic line bundles (see \cite{Yuan} and also \cite{Gubler} for the function field version), which  can be applied due to our Theorem~\ref{specialization theorem}.

\begin{proof}[Proof of Theorem~\ref{equidistribution theorem}.]
By assumption, we know $\prep(c_1, \ldots, c_{m+1})$ is a Zariski dense set of $\bA^m(\Kbar)$. Without loss of generality, it suffices to prove that for each $(\l_1,\dots, \l_m)\in\bA^m(\Kbar)$, if $c_i$ is preperiodic for $\bff_{\bflambda}$ for each $i=1,\dots, m$, then also $c_{m+1}$ is preperiodic for $\bff_{\bflambda}$.

Recall from Section~\ref{sec:specialization} the family $\Phi$ of endomorphisms of $\bP^m$   defined by
$$ \Phi([X_m:\cdots :X_1:X_0]) = \left[ X_0^d\bff\left(\frac{X_m}{X_0}\right) : \cdots :  X_0^d\bff\left(\frac{X_1}{X_0}\right): X_0^d\right].$$
As before, for each $\bflambda = (\l_1,\ldots, \l_m) \in\bA^m(\Kbar)$, we denote by $\Phi_{\bflambda}$ the corresponding endomorphism of $\bP^m$ obtained by specializing each $t_i$ to $\l_i$. For each $j=1,2$ we let
$$\bfc^{(j)} := [c_{m-1+j}: c_{m-1}: \cdots : c_1 : 1]$$
and for each $n\ge 0$ we define polynomials $A^{(j)}_{n,i}(t_1,\dots, t_m)\in \Kbar[t_1,\dots, t_m]$ (for $i=1,\dots, m$) such that
$$\Phi^n(\bfc^{(j)})=[A^{(j)}_{n,m}:\cdots :A^{(j)}_{n,1}:1].$$
More precisely, $[A^{(j)}_{0,m}:\cdots :A^{(j)}_{0,1}:1]=\bfc^{(j)}$, while for each $n\ge 0$, we have $A^{(j)}_{n+1,i}=\bff(A^{(j)}_{n,i})$.
It is easy to check that the total degree in $t_1,\dots, t_m$  is $\deg A^{(j)}_{n,i}=d^{n-1}$ for all $n\ge 1$ (and each $i=1,\dots,m$ and each $j=1,2$).

We note that if we let
$$\tilde{A}^{(j)}_{n,i}(u_1,\dots, u_{m+1}):=u_{m+1}^{d^{n-1}}\cdot A_{n,i}^{(j)}\left(\frac{u_1}{u_{m+1}},\dots, \frac{u_m}{u_{m+1}}\right)$$
(for each $j=1,2$, each $i=1,\dots, m$ and each $n\in\N$),
then the map
\begin{gather*}
\theta^{(j)}_n :\bP^m\lra \bP^m \quad \text{given by}\\
\theta^{(j)}_n (\bfu) = \left[\tilde{A}^{(j)}_{n,m}(\bfu): \cdots :\tilde{A}^{(j)}_{n,1}(\bfu) :u_{m+1}^{d^{n-1}}\right],\;\bfu = [u_1 : \cdots: u_{m+1}],
\end{gather*}
is a morphism defined over $K$. Indeed, if $u_{m+1}=0$, we have
$$\tilde{A}^{(j)}_{n,i}(u_1,\dots, u_m,0)=\left(\sum_{k=1}^m c_i^{m-k}u_k\right)^{d^{n-1}}$$
for $i=1,\dots, m-1$, and
$$\tilde{A}^{(j)}_{n,m}(u_1,\dots, u_m,0)=\left(\sum_{k=1}^m c_{m-1+j}^{m-k}u_k\right)^{d^{n-1}}.$$
Now the assumption that the $c_i$'s are distinct ensures that the above map is well-defined on $\bP^m$.  Thus,
we have an isomorphism
$$\tau^{(j)}_n:  \cO_{\bP^m}(d^{n-1})  {\tilde \lra} \left(\theta^{(j)}_n\right)^* \cO_{\bP^m}(1),$$
 given by
$$
 \tau_n^{(j)}\left(u_i^{d^{n-1}}\right)= \tilde{A}_{n,i}^{(j)}(u_1,\dots, u_{m+1})$$ for $i=1,\dots, m$, and also $\tau_n^{(j)}\left( u_{m+1}^{d^{n-1}}\right)=  u_{m+1}^{d^{n-1}}$.

We consider the following two families of
metrics corresponding to any section $s:=a_1u_1+\cdots + a_{m+1}u_{m+1}$ (with
scalars $a_i$) of the line bundle $\cO_{\bP^m}(1)$ of $\bP^m$. Using
the coordinates $t_i=\frac{u_i}{u_{m+1}}$ (for $i=1,\dots, m$)  on the
affine subset of $\bP^m$ corresponding to $u_{m+1}\ne 0$, then for each $v\in\Omega_K$,  for each $n\in
\N$ (and each $j=1,2$) we get that the metrics $\|s(\cdot )\|^{(j)}_{v,n}$ are defined as follows: 
\begin{equation}
\label{definition of the metrics s}
\|s([u_1:\cdots : u_{m+1}])\|^{(j)}_{v,n} =
\begin{cases}
\frac{\left|\sum_{k=1}^m a_ku_k\right|_v}{\max_{i=1}^m\{|\tilde{A}^{(j)}_{1,i}(u_1,\dots, u_m,0)|_v\}} & \text{if  $u_{m+1}=0$,}\bigskip\\
\frac{\left|a_{m+1}+\sum_{k=1}^m a_kt_k\right|_v}{ \sqrt[d^{n-1}]{\norm{\Phi^n(\bfc^{(j)})}_v}} & \text{if  $u_{m+1}\ne 0$.}
\end{cases}
\end{equation}

Let $\| \cdot \|'_v$ be the metric on $\cO_{\bP^m}(1)$ corresponding to the section $s=a_1u_1+\cdots + a_{m+1}u_{m+1}$
given by
\[ \| s([b_1:\cdots :b_{m+1}])\|'_v = \frac{ \left| \sum_{k=1}^{m+1} a_kb_k\right|_v}
{\max\{|b_1|_v,\ldots, |b_{m+1}|_v\}}. \]   We see then that $\|s \|^{(j)}_{v,n}$ is simply the $d^{n-1}$-th root of $\left(\tau^{(j)}_n\right)^* \left(\theta^{(j)}_n\right)^* \|
\cdot \|'_v$.  Note that the degree of $\theta^{(j)}_n$ is the same as the total degree of the
polynomials $A_{n,i}$, and thus $\deg \theta^{(j)}_n = d^{n-1}$.
Hence, for each $n$, we have that $\| s \|^{(j)}_{v,n}$ are
semipositive metrics on $\cL = \cO_{\bP^m}(1)$.  Following \cite{Yuan} (in the case of number fields) and \cite{Gubler} (in the case of function fields), we let ${\overline \cL}^{(j)}_n$
denote the algebraic adelic metrized line bundle corresponding to the collection of metrics $\| s\|^{(j)}_{v,n}$.

Let $j=1,2$. Clearly,  $\left\{\log\| s\|^{(j)}_{v,n}\right\}_n$ converges uniformly on the hyperplane at infinity (defined by $u_{m+1} = 0$) from $\bP^m$ since there is no dependence on $n$ in this case. Lemmas~\ref{reduction to finitely many places} and \ref{analysis of bad places} yield that  $\left\{\log\| s\|^{(j)}_{v,n}\right\}_n$ converges uniformly also when $u_{m+1}\ne 0$. Furthermore (as shown by Lemma~\ref{reduction to finitely many places}), for all but finitely many places of $K$, the metrics $\| s \|^{(j)}_{v,n}$ do not vary with $n$. For each $v\in\Omega_K$, we let $\|s\|^{(j)}_v$ be the metric which is the limit of the metrics $\| s \|^{(j)}_{v,n}$. We denote by $\overline{\cL}^{(j)}$ the corresponding ad\'elic metrized line bundles $\left(\cO_{\bP^m}(1), \{\|s\|^{(j)}_{v}\}\right)$.

Let $Q\in\bP^m(\Kbar)$. Let $s$ be a section of $\cL$ as above such that $s(Q)\ne 0$; we define the height $\hhat_{\overline{\cL}^{(j)}}(Q)$ associated to the metrized lline bundle $\overline{\cL}^{(j)}$  as follows. We let $L$ be a normal, finite extension of $K$ such that $Q\in\bP^m(K)$ and then define:
\begin{equation}\label{height-def}
\hhat_{\overline{\cL}^{(j)}}(Q):= \sum_{v\in\Omega_K} \frac{N_v}{[L:K]} \sum_{\sigma\in\Gal(L/K)} -\log \| s(\sigma(Q)) \|^{(j)}_v.
\end{equation}
By the definition of the above ad\'elic metrics, we have (see also \cite[(9.0.8)]{metric}) for each $j=1,2$:
\begin{equation}
\label{connection heights}
\hhat_{\overline{\cL}^{(j)}}([\l_m:\cdots:\l_1:1])= d\cdot \hhat_{\Phi_{\bflambda}}(\bfc^{(j)}).
\end{equation}
By our assumption, there exists a Zariski dense set of points $\bflambda \in \bA^m(\Kbar)$ such that
\begin{equation}
\label{infinitely many sections 0}
\hhat_{\overline{\cL}^{(1)}}(\bflambda)=\hhat_{\overline{\cL}^{(2)}}(\bflambda)=0.
\end{equation}
On the other hand, we note that
\begin{equation}
\label{norm j 1 2}
\norm{\theta_1^{(j)}([u_1:\cdots:u_m:0])}_v = \max\left\{\left|\sum_{k=1}^m c^{m-k}_{m-1+j} u_k\right|_v, \max_{i=1}^{m-1}\left|\sum_{k=1}^m c_i^{m-k}u_k\right|_v\right\}
\end{equation}
for  each place $v\in \Omega_K$.
Then for each point  $[u_1:\cdots : u_m:0]$ on the hyperplane at infinity of $\bP^m$ such that
\begin{equation}
\label{equality of sections}
\sum_{k=1}^m c_m^{m-k}u_k = \sum_{k=1}^m c_{m+1}^{m-k}u_k,
\end{equation}
\eqref{norm j 1 2} yields that
\begin{equation}
\label{norm j 1}
\norm{\theta_1^{(1)}([u_1:\cdots:u_m:0])}_v =\norm{\theta_1^{(2)}([u_1:\cdots:u_m:0])}_v.
\end{equation}
On the other hand, the definition \eqref{definition of the metrics s} of the metric $\|\cdot\|_v^{(j)}$  at any point $[u_1:\cdots :u_m:0]$ on the hyperplane at infinity and for any section $s:=a_1u_1+\cdots +a_{m+1}u_{m+1}$ gives
\begin{align*}
\|s([u_1:\cdots :u_m:0])\|^{(j)}_v& =\frac{\left|\sum_{k=1}^m a_ku_k\right|_v}{\max_{i=1}^m\{|\tilde{A}^{(j)}_{1,i}(u_1,\dots, u_m,0)|_v\}}\\
 & = \frac{\left|\sum_{k=1}^m a_ku_k\right|_v}{\norm{\theta_1^{(j)}([u_1:\cdots:u_m:0])}_v}
 \end{align*}
for each $j=1,2$. So, for a point $[u_1:\cdots :u_m:0]$ satisfying \eqref{equality of sections},  equality \eqref{norm j 1} yields
\begin{equation}
\label{equality of sections 0}
\|s([u_1:\cdots :u_m:0])\|^{(1)}_v=\|s([u_1:\cdots :u_m:0])\|^{(2)}_v.
\end{equation}
Combining \eqref{infinitely many sections 0} and \eqref{equality of sections 0} allows us to use \cite[Corollary~4.3]{metric} and conclude that for \emph{all} $\l_1,\dots,\l_m\in\Kbar$ we have
\begin{equation}
\label{equality of the heights}
h_{\overline{\cL}^{(1)}}(\bflambda)=h_{\overline{\cL}^{(2)}}(\bflambda).
\end{equation}
Strictly speaking, \cite[Corollary~4.3]{metric} was stated only for metrized line bundles defined over $\Qbar$ since the authors employed in that paper Yuan's equidistribution theorem from \cite{Yuan}. However, using \cite[Theorem~1.1]{Gubler} and arguing identically as in the proof of \cite[Corollary~4.3]{metric} one can extend the result from number fields to any function field of characteristic $0$.

Using \eqref{equality of the heights} coupled with \eqref{connection heights}, we obtain that $\hhat_{\Phi_{\bflambda}}(\bfc^{(1)})=0$ if and only if $\hhat_{\Phi_{\bflambda}}(\bfc^{(2)})=0$.

Assume now that $K$ is a number field, i.e., that each $c_i\in\Qbar$. Then, as shown in \cite{Call-Silverman}, a point is preperiodic under the action of $\Phi_{\bflambda}$ if and only if its canonical height equals $0$. On the other hand, in general, a point $[a_1:\cdots :a_m:1]\in\bP^m(\Qbar)$ is preperiodic under the action of $\Phi_{\bflambda}$ if and only if each $a_i$ is preperiodic for the action of $\bff_{\bflambda}$. Therefore, we obtain that for each $\bflambda\in\bA^m(\Qbar)$, if each $c_i$ (for $i=1,\dots, m$)  is  preperiodic for $\bff_{\bflambda}$, then also $c_{m+1}$ is preperiodic for $\bff_{\bflambda}$.

So, from now on, assume that not all $c_i$ are contained in $\Qbar$. It is still true that $\hhat_{\Phi_{\bflambda}}(\bfc^{(1)})=0$ if and only if $\hhat_{\bff_{\bflambda}}(c_i)=0$ for $i=1,\dots, m$. Arguing  similarly for $\bfc^{(2)}$, we get that for a $\bflambda\in\bA^M(\Kbar)$ such that $\hhat_{\Phi_{\bflambda}}(\bfc^{(1)})=\hhat_{\Phi_{\bflambda}}(\bfc^{(2)})=0$, we have that each $\hhat_{\bff_{\bflambda}}(c_i)=0$ for $i=1,\dots, m+1$. But $\bff_{\bflambda}$ is a polynomial in normal form and therefore, it is isotrivial if and only if each one of its coefficients are in the constant field, i.e. they are contained in $\Qbar$ (see Proposition~\ref{normal form proposition}). But if this happens then  we cannot have that each $\hhat_{\bff_{\bflambda}}(c_i)=0$ since not all $c_i$ are in the constant field. In conclusion, $\bff_{\bflambda}$ is not isotrivial, and then by Lemma~\ref{Benedetto}, we conclude that $\hhat_{\bff_{\bflambda}}(c_i)=0$ if and only if $c_i$ is preperiodic under the action of $\bff_{\bflambda}$. Hence, if $c_i$ is preperiodic for each $i=1,\dots, m$, then also $c_{m+1}$ is preperiodic under the action of $\bff_{\bflambda}$.
\end{proof}


\section{Proof of Theorem~\ref{main result}}
\label{proof section}
We work under the hypotheses of Theorem~\ref{main result}.

Because the starting points $c_i$ are all distinct, we may assume $c_1\ne 0$. Let now $V\subset \bP^2$ be the line which is the Zariski closure in $\bP^2$ of the affine line containing all $\bflambda = (\l_1,\l_2)\in\bA^2(\Kbar)$ such that $c_1$ is a fixed point for $\bff_{\bflambda}(x) = x^d+\l_1 x + \l_2$. This last condition is equivalent with asking that $c_1^d+c_1\l_1 + \l_2 = c_1$, or in other words, $V$ is the line in $\bP^2$ whose intersection with the affine plane containing all points in $\bP^2$ with  a nonzero last coordinate is the line $(t, \alpha+\beta t)$, where $\alpha:=c_1-c_1^d$ and $\beta :=-c_1\ne 0$.

Let $\bfg_t(x):=x^d+tx+(\alpha +\beta t)$ be a $1$-parameter family of degree $d$ polynomials.
Note that $\alpha +\beta t \ne 0$ (because $\beta\ne 0$). By Theorem~\ref{equidistribution theorem},
we know that for each $t\in\Kbar$, $c_2$ is preperiodic for $\bfg_t$ if and only if $c_3$ is preperiodic for $\bfg_t$. It follows from~\cite[Theorem~1.3]{BD-preprint} that there exists a polynomial $\bfh_t$ commuting with an iterate of $\bfg_t$, and there exist $m,n\in\N$ such that
\begin{equation}
\label{commutation equation}
\bfg^m_t(c_2)=\bfh_t\left(\bfg^n_t(c_3)\right).
\end{equation}

We claim that if $\bfh_t\in\Kbar[t,x]$ is any (non-constant)  polynomial such that $ \bfh_t$ commutes with an iterate of $\bfg_t$, then $\bfh_t = \bfg_t^\ell$ for some $\ell\ge 1$. This statement follows from \cite[Proposition~2.3]{Nguyen}. First, since $\bfg_t$ is in normal form, the only linear polynomials commuting with $\bfg_t$ are of the form $\gamma x$, and because $\alpha + \beta t\ne 0$, then $\gamma=1$. Secondly, according to \cite[Proposition~2.3]{Nguyen}, the  non-linear polynomial $\bfh_t$ of smallest degree commuting with $\bfg_t$ must satisfy the condition that $\bfh_t^e=\bfg_t$ for some positive integer $e$. Due to the shape of $\bfg_t$, the only possibility is $e=1$, which yields our claim that the only polynomials commuting with $\bfg_t$ are of the form $\bfg_t^\ell$ for $\ell\ge 1$.

Therefore, in \eqref{commutation equation} we can take $\bfh_t(x)=x$; hence
\begin{equation}
\label{commutation equation 2}
\bfg^m_t(c_2)=\bfg^n_t(c_3).
\end{equation}

Now, for each $c\in\Kbar$, if $c\ne c_1$ then $\deg_t(\bfg_t(c))=1$ and then a simple induction proves that $\deg_t(\bfg^\ell_t(c))=d^{\ell-1}$ for all $\ell\in\N$. Therefore, \eqref{commutation equation 2} yields that $m=n$. The equality $\bfg_t^m(c_2)=\bfg_t^m(c_3)$ yields in particular that the leading coefficients of the polynomials $\bfg_t^m(c_2)$ and $\bfg_t^m(c_3)$ are the same, i.e.
\begin{equation}
\label{relation between the c's}
(c_2-c_1)^{d^{m-1}}=(c_3-c_1)^{d^{m-1}}.
\end{equation}

It is immediate to see that for $m\ge 2$ we have
$$\bfg_t^{m-1}(x)=x^{d^{m-1}} + d^{m-2}t\cdot x^{d^{m-1}-d+1}+\text{ lower order terms in $x$.}$$

The rest of the argument is split into two cases depending on whether $d\ge 4$, or $d=3$.

{\bf Assume now that $d\ge 4$.}

Noting that $\bfg_t(c_2)=t(c_2-c_1)+c_2^d+c_1-c_1^d$, then for $m\ge 2$ we get that
\begin{align*}
\bfg_t^{m}(c_2) & = g_t(c_2)^{d^{m-1}} + d^{m-2}t \cdot g_t(c_2)^{d^{m-1}-d+1}+ R_m(t) \\
& = t^{d^{m-1}}\cdot (c_2-c_1)^{d^{m-1}} \\
&+ d^{m-1}\cdot (c_2-c_1)^{d^{m-1}-1}(c_2^d+c_1-c_1^d)\cdot  t^{d^{m-1}-1}+ O\left(t^{d^{m-1}-2}\right)
\end{align*}
where $ R_m(t)$ is a polynomial in $t$ of degree at most $d^{m-1}-d+1$.
In the above computation we used the fact that $d>3$ and therefore the second leading term of $\bfg_t^m(c_2)$ is indeed $d^{m-1}\cdot (c_2-c_1)^{d^{m-1}-1}(c_2^d+c_1-c_1^d)\cdot t^{d^{m-1}-1}$. Also, the first equality in the above expansion of $\bfg_t^m(c_2)$ follows by induction on $m$.
 This follows immediately since $\bfg_t(x)=x^d+tx + (c_1-c_1^d-tc_1)$ and so,
\begin{align*}
\bfg_t^{m+1}(c_2)
& = \bfg_t(\bfg_t^m(c_2))\\
& = \bfg_t^m(c_2)^d + t\bfg_t^m(c_2)+ (c_1-c_1^d-tc_1)\\
& = \left(g_t(c_2)^{d^{m-1}} + d^{m-2}t\cdot g_t(c_2)^{d^{m-1}-d+1}+ R_m(t)\right)^d \\
& + t \cdot \left(g_t(c_2)^{d^{m-1}} + d^{m-2}t g_t(c_2)^{d^{m-1}-d+1}+ R_m(t)\right) + (c_1-c_1^d-c_1 t)
\end{align*}
Using the induction hypothesis (and that $d\ge 3$), we know that
$$\deg_t g_t(c_2)^{d^{m-1}} > \deg_t t\cdot g_t(c_2)^{d^{m-1}-d+1} > \deg_t R_m(t) .$$
We  thus conclude that
\begin{align*}
\bfg_t^{m+1}(c_2)  & = g_t(c_2)^{d^{m}}
 + d^{m-1}t\cdot g_t(c_2)^{d^{m}-d+1} + O(t^{d^{m}-d+1})\\
& + t\cdot g_t(c_2)^{d^{m-1}} + O(t^{d^{m-1}})
\end{align*}
Clearly, $d^m-d+1 > d^{m-1}$ for $m\ge 2$ (because $d\ge 3$), which yields the desired claim that
\begin{align*}
\bfg_t^m(c_2)& = g_t(c_2)^{d^{m-1}} + d^{m-2}t\cdot g_t(c_2)^{d^{m-1}-d+1}+ R_m(t) \\
& =  (c_2-c_1)^{d^{m-1}}  \cdot t^{d^{m-1}}\\
&+ d^{m-1}\cdot (c_2-c_1)^{d^{m-1}-1}(c_2^d+c_1-c_1^d)\cdot  t^{d^{m-1}-1}+ O\left(t^{d^{m-1}-2}\right)
\end{align*}
Hence $\bfg_t^m(c_2)=\bfg_t^m(c_3)$ yields not only \eqref{relation between the c's} but also that
\begin{equation}
\label{relation between the c's 2}
(c_2-c_1)^{d^{m-1}-1}\cdot (c_2^d+c_1-c_1^d)=(c_3-c_1)^{d^{m-1}-1}\cdot (c_3^d+c_1-c_1^d).
\end{equation}
Equations \eqref{relation between the c's} and \eqref{relation between the c's 2} yield that there exists $u\in K$ such that
\begin{equation}
\label{relation between the c's 3}
u:=\frac{c_2^d+c_1-c_1^d}{c_2-c_1}=\frac{c_3^d+c_1-c_1^d}{c_3-c_1}.
\end{equation}
Combining \eqref{relation between the c's} and \eqref{relation between the c's 3}, we get that
\begin{equation}
\label{relation between the c's 4}
g_t(c_2)^{d^{m-1}}= g_t(c_3)^{d^{m-1}}.
\end{equation}
Then using \eqref{relation between the c's 4} and the expansion of $\bfg_t^m(c_2)=\bfg_t^m(c_3)$ in terms of powers of $t$, we get that
\begin{align*}
& d^{m-2}t\cdot \left(t(c_2-c_1)+(c_2^d+c_1-c_1^d)\right)^{d^{m-1}-d+1} \\
& = d^{m-2}t\cdot \left(t(c_3-c_1)+(c_3^d+c_1-c_1^d)\right)^{d^{m-1}-d+1} +O(t^{d^{m-1}-d +1}).
\end{align*}
This yields that
\begin{equation}
\label{relation between the c's 6}
(c_2-c_1)^{d^{m-1}-d+1}=(c_3-c_1)^{d^{m-1}-d+1}.
\end{equation}
Since $\gcd\left(d^{m-1}, d^{m-1}-d+1\right)=1$, \eqref{relation between the c's} and \eqref{relation between the c's 6} yield that
$$c_2-c_1=c_3-c_1,$$
i.e., that $c_2=c_3$, contradiction. This concludes our proof when $d\ge 4$.

{\bf Assume now that $d=3$.}

In this case we employ a slightly different argument since the second leading term in $\bfg_t^m(c_2)$ involves also contribution from $$d^{m-2}t\cdot \left(t(c_2-c_1)+c_2^d+c_1-c_1^d\right)^{d^{m-1}-d+1}.$$ Instead we use additional specialization of $\bff(x)=x^d+t_1x+t_2$ along other lines in the moduli space $\bA^2$ with the property that $c_2$ (and $c_3$) are fixed along these other lines. This allows us to derive additional relations between the $c_i$'s similar to \eqref{relation between the c's}.

Without loss of generality, we may also assume $c_2\ne 0$ (because all three numbers $c_1,c_2,c_3$ are distinct). Then the exact same argument as above used for deriving the equation \eqref{relation between the c's} (applied this time to the line in the parameter space along which $c_2$ is a  fixed point) yields that
\begin{equation}
\label{relation between the c's 7}
(c_1-c_2)^{3^{\ell-1}}=(c_3-c_2)^{3^{\ell-1}},
\end{equation}
for some $\ell\in\N$. Using \eqref{relation between the c's} and \eqref{relation between the c's 7}, at the expense of replacing $\ell$ by a larger number, we may assume
\begin{equation}
\label{relation between the c's 8}
(c_2-c_1)^{3^{\ell}}=(c_3-c_1)^{3^{\ell}}\text{ and }(c_1-c_2)^{3^{\ell}}=(c_3-c_2)^{3^{\ell}}.
\end{equation}
We split now the analysis depending on whether $c_3$ is also nonzero, or $c_3=0$.

{\bf Case 1.} $c_3\ne 0$.

In this case, we can apply (a third time) the above argument, this time for the curve in the parameter space along which $c_3$ is a fixed point, and therefore conclude (at the expense of replacing $\ell$ by a larger integer) that
\begin{equation}
\label{relation between the c's 9}
(c_2-c_1)^{3^{\ell}}=(c_3-c_1)^{3^{\ell}}\text{,  }(c_1-c_2)^{3^{\ell}}=(c_3-c_2)^{3^{\ell}}\text{ and }(c_1-c_3)^{3^{\ell}}=(c_2-c_3)^{3^{\ell}}.
\end{equation}
But then $(c_2-c_3)^{3^{\ell}}=(c_3-c_2)^{3^{\ell}}$, which yields that $c_2=c_3$, contradiction.

{\bf Case 2.} $c_3=0$.

Under this assumption, we rewrite \eqref{relation between the c's 8} as follows:
\begin{equation}
\label{relation between the c's 10}
(c_2-c_1)^{3^{\ell}}=(-c_1)^{3^{\ell}}\text{ and }(c_1-c_2)^{3^{\ell}}=(-c_2)^{3^{\ell}}.
\end{equation}
Hence there exist $3^{\ell}$-th roots of unity $\zeta_1$ and $\zeta_2$ such that $c_1-c_2=-\zeta_1 c_2$ and $c_2-c_1=-\zeta_2 c_1$. So, $c_1=(1-\zeta_1)c_2$ and $c_2=(1-\zeta_2)c_1$ and because $c_1c_2\ne 0$, we conclude that $(1-\zeta_1)(1-\zeta_2)=1$, i.e. $-\zeta_1-\zeta_2+\zeta_1\zeta_2=0$. Hence $\zeta_2(\zeta_1-1)=\zeta_1$ and thus also $\zeta_1-1$ is a  $3^{\ell}$-th root of unity. Now, the only roots of unity $\zeta$ with the property that also $\zeta-1$ is a root of unity are $\zeta = \frac{1\pm \sqrt{-3}}{2}$. However, $\frac{1\pm \sqrt{-3}}{2}$ is a primitive $6$-th root of unity and not a $3^{\ell}$-th root of unity, contradiction.

This concludes the proof of Theorem~\ref{main result}.







\end{document}